\documentclass[reqno]{amsart}
\usepackage{amsmath, amsthm, amssymb, mathabx}
\begin{document}
\newtheorem{theorem}{Theorem}[section]
\newtheorem{corollary}{Corollary}
\newtheorem*{main}{Main Theorem}
\newtheorem{lemma}[theorem]{Lemma}
\newtheorem{proposition}{Proposition}
\newtheorem{conjecture}{Conjecture}
\newtheorem*{problem}{Problem}
\theoremstyle{definition}
\newtheorem{definition}[theorem]{Definition}
\newtheorem{remark}{Remark}
\newtheorem*{notation}{Notation}
\newcommand{\ep}{\varepsilon}
\newcommand{\eps}[1]{{#1}_{\varepsilon}}

\numberwithin{equation}{section}

%
%
 \thanks{}
\title[Schr\"{o}dinger-Kirchhoff-Poisson type systems]{Schr\"{o}dinger-Kirchhoff-Poisson type systems}

\author[C. J. Batkam]{Cyril Joel Batkam}
\address{Cyril Joel Batkam
 \newline
The Fields Institute,
\newline
Toronto, Ontario, M5T 3J1, CANADA.}
\email{cbatkam@fields.utoronto.ca, cyrilbatkam@gmail.com}
\author[J.R. Santos J\'unior]{Jo\~ao R. Santos J\'unior}
\address{Jo\~ao R. Santos J\'unior
 \newline
Universidade Federal do Par\'a,
\newline
Faculdade de Matem\'atica
\newline
CEP 66075-110, Bel\'em, Par\'a, BRAZIL.}
\email{joaojunior@ufpa.br, joaorodrigues\_mat@hotmail.com}

\maketitle
\begin{abstract}
In this article we study the existence of solutions to the system
\begin{equation*}
  \left\{
      \begin{array}{ll}
      -\left(a+b\int_{\Omega}|\nabla u|^{2}\right)\Delta u +\phi u= f(x, u)   &\text{in }\Omega \hbox{} \\
       -\Delta \phi= u^{2} &\text{in }\Omega  \hbox{} \\
         u=\phi=0&\text{on }\partial\Omega,  \hbox{} 
      \end{array}
    \right.
\end{equation*}
where  $\Omega$ is a bounded smooth domain of $\mathbb{R}^N$ ($N=1,2$ or $3$), $a>0$, $b\geq0$, and $f:\overline{\Omega}\times \mathbb{R}\to\mathbb{R}$ is a continuous function which is $3$-superlinear. By using some variants of the mountain pass theorem established in this paper, we show the existence of three solutions: one positive, one negative, and one which changes its sign. Furthermore, in case $f$ is odd with respect to $u$ we obtain an unbounded sequence of sign-changing solutions.
\end{abstract}

\section{Introduction}

In this article, we study the existence of signed and sign-changing solutions to the following Schr\"{o}dinger-Kirchhoff-Poisson system
\begin{equation}\label{1}\tag{SKP}
 \left\{
      \begin{array}{ll}
      -\left(a+b\int_{\Omega}|\nabla u|^{2}\right)\Delta u +\phi u= f(x, u)   &\text{in }\Omega \hbox{} \\
       -\Delta \phi= u^{2} &\text{in }\Omega  \hbox{} \\
         u=\phi=0&\text{on }\partial\Omega,  \hbox{} 
      \end{array}
    \right.
\end{equation}
where  $\Omega$ is a bounded smooth domain of $\mathbb{R}^N$ with $N=1,2$ or $3$; $f$ is a continuous function satisfying some conditions we will precise later, $a>0$ and $b\geq 0$.

When $a=1$ and $b=0$, \eqref{1} reduces to the boundary value problem
\begin{equation}\label{SP}
\left \{ \begin{array}{ll}
-\Delta u +\phi u= f(x, u) & \mbox{in $\Omega$,}\\
-\Delta \phi= u^{2} & \mbox{in $\Omega$,}\\
u=\phi=0 & \mbox{on $\partial\Omega$.}
\end{array}\right.
\end{equation}
Knowledge of the solutions of system \eqref{SP} is relevant in the study of stationary solutions $\psi(x, t)=e^{-it}u(x)$ to the nonlinear parabolic Schr\"{o}dinger-Poisson system
\begin{equation}\label{PSP}
\left \{ \begin{array}{ll}
-i\frac{\partial\psi}{\partial t}=-\Delta \psi +\phi(x)\psi-|\psi|^{p-2}\psi & \mbox{in $\Omega$,}\\
-\Delta \phi= |\psi|^{2} & \mbox{in $\Omega$,}\\
\psi=\phi=0 & \mbox{on $\partial\Omega$.}
\end{array}\right.
\end{equation}
The first equation in \eqref{PSP}, called Schr\"{o}dinger equation, describes quantum (non-relativistic) particles interacting with the eletromagnetic field generated by the motion. An interesting class of Schr\"{o}dinger equations is when the potential $\phi(x)$ is determined by the charge of wave function itself, that is, when the second equation in \eqref{PSP} (Poisson equation) holds. For more details about the physical relevance of the Schr\"{o}dinger-Poisson system, we refer to \cite{AR,BF,Ruiz}.

System \eqref{SP} has been extensively studied after the seminal work of Benci and Fortunato \cite{BF}. Many important results concerning existence and non existence of solutions, multiplicity of solutions, least energy solutions, radial and non radial solutions, semiclassical limit and concentrations of solution have been obtained. See for instance \cite{AS,AR,AP,Batkam1,CV,Coclite,RS} and the references therein.

On the other hand, considering just the first equation in \eqref{1} with the potential equal to zero, we have the problem
\begin{equation}\label{KP}
\left \{ \begin{array}{ll}
-\left(a+b\int_{\Omega}|\nabla u|^{2}\right)\Delta u= f(x, u) & \mbox{in $\Omega$,}\\
u=0 & \mbox{on $\partial\Omega$,}
\end{array}\right.
\end{equation}
which represents the stationary and N-dimensional version of the Kirchhoff model \cite{kirchhoff} for small transverse vibrations of an elastic string by considering the effect of the changing in the length during the vibrations. In fact, since the length of the string is variable during the vibrations, the tension changes with the time and depends of the $L^{2}$ norm of the gradient of the displacement $u$. More precisely, we have $a=P_{0}/h$ and $b=E/2L$, where $L$ is the length of the string, $h$ is the area of cross-section, $E$ is the Young modulus of the material and $P_{0}$ is the initial tension. Problem \eqref{KP} is called nonlocal because of the presence of the term $\int_{\Omega}|\nabla u|^{2}dx$ which implies, when $b\neq0$, that the equation in \eqref{KP} is no longer a pointwise identity. This phenomenon causes some mathematical difficulties which make the study of such class of problems particularly interesting. Some existence and multiplicity results on Kirchhoff type problems can be found in \cite{alvescorreama05,Anelo1,Anelo2,Batkam1,MR,ma} and the references therein.


An important fact about system \eqref{1} is that it can be converted into a bi-nonlocal problem of the Schr\"{o}dinger-Kirchhoff type. More precisely, by using standard arguments as those in \cite{BF}, one can show that $(u,\phi)\in H_0^1(\Omega)\times H_0^1(\Omega)$ is a weak solution of \eqref{1} if, and only if, $\phi=\phi_u$ and $u$ is a weak
solution to the following Schr\"{o}dinger-Kirchhoff type system
\begin{equation}\label{2}\tag{SK}
\left \{ \begin{array}{ll}
-\left(a+b\int_{\Omega}|\nabla u|^{2}\right)\Delta u +\phi_{u} u= f(x, u)  & \mbox{in $\Omega$,}\\
u=0 & \mbox{on $\partial\Omega$},
\end{array}\right.
\end{equation}
where $\phi_u$ is the unique element of $H_0^1(\Omega)$, given by the Lax Milgram Theorem, such that $-\Delta \phi_{u}= u^{2} $. We will then concentrate our efforts in the study of \eqref{2}.

In recent years, Schr\"{o}dinger-Kirchhoff problems like \eqref{2}, with $\phi=\phi(x)$ depending only on $x$, have received great attention of the mathematical community. In \cite{HZ}, by using Lusternik-Schnirelmann theory and minimax methods,  He and Zou  proved a result of multiplicity and concentration behavior of positive solutions for the following equation
\begin{equation}\label{HZ}
\left \{ \begin{array}{ll}
-\left(\varepsilon^{2}a+b\varepsilon\int_{\mathbb{R}^3}|\nabla u|^{2}\right)\Delta u + V(x)u=f(u) & \mbox{in $\mathbb{R}^3$}\\
u \in H^{1}(\mathbb{R}^3),
\end{array}\right.
\end{equation}
by assuming, among others, that $f\in C^{1}(\mathbb{R}^3)$ is subcritical and 3-superlinear, and that the potential $V$ satisfies:
\begin{equation*}
V_{\infty}=\liminf_{|x|\to
\infty}V(x)>V_{0}=\inf_{\mathbb{R}^3}V(x)>0.
\end{equation*}
In \cite{WTXZ}, Wang et al. replaced the second member of \eqref{HZ} by $\lambda f(u)+u^5$ and they obtained, assuming only that $f$ is continuous, multiple positive solutions when the parameter $\lambda>0$ is large enough. Their approach combines  Lusternik-Schnirelmann theory, minimax methods, and the Nehari manifold method. More recently,  Figueiredo and Santos J\'{u}nior \cite{FSJ} obtained, using penalization method and the Nehari manifold approach, a multiplicity and concentration result of positive solutions for the problem
\begin{equation*}
\left \{ \begin{array}{ll}
-M\left(\frac{1}{\varepsilon}\int_{\mathbb{R}^3}|\nabla u|^{2}+\frac{1}{\varepsilon^{3}}\int_{\mathbb{R}^3}V(x)u^{2}\right)\left[-\varepsilon^{2}\Delta u + V(x)u \right]=f(u) & \mbox{in $\mathbb{R}^3$},\\
u \in H^{1}(\mathbb{R}^3),
\end{array}\right.
\end{equation*}
with $f$ 3-superlinear and only continuous, $M$ a general continuous function, and $V$ satisfying the condition: {\em for each $\delta>0$ there is a bounded and Lipschitz domain $\Omega \subset \mathbb{R}^3$ such that $V_{0}< \min_{\partial \Omega}V$, with $\Pi=\{x\in \Omega: V(x)=V_{0}\}\neq \emptyset$ and $\Pi_{\delta}=\{ x\in \mathbb{R}^3: dist(x, \Pi)\leq \delta\}\subset \Omega$.} We also refer to \cite{FISJ, LLS, NW, Wu} for related results.

Motivated by the previous work, we study the existence of solutions to the system \eqref{1} or, equivalently, to the system  \eqref{2}. As far as we know, this is the first paper to investigate a bi-nonlocal problem of this type. We emphasize that the combined effects of the two nonlocal terms it contains make  problem \eqref{2} an interesting variational problem. We will have to circumvent some new difficulties in order to decide the sign of the solutions.
 To enunciate our main result, we first require some conditions on the nonlinear term $f$:
\begin{enumerate}
  \item[$(f_1)$] $f\in C\big(\overline{\Omega}\times\mathbb{R},\mathbb{R}\big)$ and there exists a constant $c>0$ such that
\begin{equation*}
    |f(x,t)|\leq c\big(1+|t|^{p-1}\big),\quad\text{where }4<p<6;
\end{equation*}
  \item[$(f_2)$] $f(x,t)=\circ(|t|)$, uniformly in $x\in \overline{\Omega}$, as $u\to0$;
  \vspace{0.3cm}
  \item[$(f_3)$] there exists $\mu>4$ such that $0<\mu F(x,t)\leq tf(x,t)$ for all $t\neq0$ and for all $x\in\overline{\Omega}$, where $F(x,t)=\int_0^t f(x,s)ds$.
 \end{enumerate}
We say that the couple $(u,\phi)$ is a sign-changing solution of \eqref{1} if $u$ changes its sign. 
 Our main result reads as follows:
  \begin{main}\label{result1}
Let $a>0$ and $b\geq0$.  If $(f_{1})-(f_3)$ hold, then problem \eqref{1} has at least three solutions: one positive, one negative, and one sign-changing. If moreover $f$ is odd with respect to its second variable, then problem \eqref{1} has infinitely many sign-changing solutions.
 \end{main}
Our approach in proving this theorem is variational and relies on the application of three critical point theorems. The first one is a new version of the mountain pass theorem established in this paper. More precisely, using the quantitave deformation lemma introduced in \cite{Batkam1}, we derive a variant of the mountain pass theorem on cones which yields positive and negative Palais-Smale sequences for the energy functional associated to \eqref{2}. The second critical point theorem is a sign-changing version of the mountain pass theorem, also established in this paper, which guarantees the existence of a sign-changing solution of mountain pass type to \eqref{2}. The main feature of this result is a new characterization of the mountain pass level introduced recently in \cite{LiuLiuWang15}. However, we point out that the critical point theorem in \cite{LiuLiuWang15} cannot be used in our situation because the auxiliary operator constructed in Section \ref{section3} is not compact if it is defined on an infinite dimensional vector space. Finally, the third critical point theorem  is a version of the symmetric mountain pass theorem, established recently in \cite{Batkam1}, which will be used, in case $f$ is odd in $u$, to ensure the existence of infinitely many high energy sign-changing solutions to \eqref{2}.

 The paper is organized as follows: In Section \ref{section2}, we state and prove the abstract results. In Section \ref{section3}, we provide the proof of Theorem \ref{result1}, which is divided into three parts.

\setcounter{equation}{0}
\section{Critical point theorems}\label{section2}
In this section, we provide some critical point theorems which are interesting by themselves and can be used in many other situations.

 Let $J$ be a $C^1$-functional defined on a Hilbert space $X$ of the form
\begin{equation}\label{space}
X:=\overline{\oplus_{j=1}^\infty X_j},\quad\text{with } \dim X_j<\infty.
\end{equation}
We introduce for $m>2$ the following notations:
\begin{equation*}
Y_m:=\oplus_{j=1}^m X_j,\quad J_m:=J|_{Y_m},
\end{equation*}
\begin{equation*}
K_m:=\big\{u\in Y_m\,;\, J'_m(u)=0\big\}\quad\text{and}\quad E_m:=Y_m\backslash K_m.
\end{equation*}
Let $P_m$ be a closed convex cone of $Y_m$. We set for $\mu_m>0$
\begin{equation*}
\pm D_m^0:=\big\{u\in Y_m\,|\, dist\big(u,\pm P_m\big)<\mu_m\big\}.
\end{equation*}
We will also denote the $\alpha$-neighborhood of $W\subset Y_m$ by
\begin{equation*}
V_\alpha(W):=\big\{u\in Y_m\,|\, dist(u,W)\leq\alpha\big\},\quad\forall\alpha>0.
\end{equation*}
\par We consider the following situation:
\begin{enumerate}
\item[$(A_0)$] There exists a locally Lipschitz continuous vector field $B:E_m\to Y_m$ \big($B$ odd if $J$ is even\big) such that:
\begin{itemize}
\item[(i)] $B\big((\pm D_m^0)\cap E_m\big)\subset \pm D_m^0$;
\item[(ii)] there exists a constant $\alpha_1>0$ such that $\big<J'_m(u),u-B(u)\big>\geq\alpha_1\|u-B(u)\|^2$, for any $u\in E_m$;
\item[(iii)] for $\rho_{1}<\rho_{2}$ and $\alpha>0$, there exists $\beta>0$ such that $\|u-B(u)\|\geq\beta$ if $u\in Y_m$ is such that $J_m(u)\in[\rho_{1},\rho_{2}]$ and $\|J'_m(u)\|\geq\alpha$.
\end{itemize}
\end{enumerate}

We have the following quantitative deformation lemma.
\begin{lemma}[see \cite{Batkam1}, Lemma $2.3$]\label{deformationlemma}
Let $J\in C^1(X,\mathbb{R})$. Assume that for all $m>2$ there exists $\mu_m>0$ such that the condition $(A_0)$ is satisfied. Let $c\in\mathbb{R}$, $\varepsilon_0>0$ and $W\subset Y_m$ \big(with $-W=W$ if $J$ is even\big) such that
\begin{equation}\label{one}
\forall u\in J_m^{-1}\big([c-2\varepsilon_0,c+2\varepsilon_0]\big)\cap V_{\frac{\mu_m}{2}}(W)\,:\, \|J_m'(u)\|\geq\varepsilon_0.
\end{equation}
Then for some $\varepsilon\in]0,\varepsilon_0[$ there exists $\eta\in C\big([0,1]\times Y_m,Y_m\big)$ such that:
\begin{enumerate}
\item[(i)] $\eta(t,u)=u$ for $t=0$ or $u\notin J_m^{-1}\big([c-2\varepsilon,c+2\varepsilon]\big)$;
\item[(ii)] $\eta\big(1,J_m^{-1}(]-\infty,c+\varepsilon])\cap W\big)\subset J_m^{-1}\big(]-\infty,c-\varepsilon]\big)$;
\item[(iii)] $J_m\big(\eta(\cdot,u)\big)$ is not increasing, for any $u$;
\item[(iv)] $\eta\big([0,1]\times (\pm D^0_m)\big)\subset \pm D^0_m$;
\item[(v)] If $J$ is even then $\eta(t,\cdot)$ is odd, for any $t\in[0,1]$.
\end{enumerate} 
\end{lemma}
To find positive and negative solutions we will use the following version of the mountain pass theorem in cones.
\begin{theorem}\label{mountain}
Let $J\in C^1(X,\mathbb{R})$. Assume that for any $m>2$ there exists $\mu_m>0$ such that $(A_0)$ is satisfied. Assume also that there exist $e^\pm\in \pm P_2$ and $r>0$ such that 
\begin{equation*}
(A_1) \qquad \|e^\pm\|>r\quad\text{and}\quad \rho:=\inf_{\substack{u\in X\\ \|u\|=r}}J(u)>\delta:=\max\{J(0),J(e^\pm)\}.
\end{equation*}
Then there exist sequences $\{u^{\pm}_{m,n}\}_n\subset \overline{\pm D^0_m}$ such that 
\begin{equation*}
\lim_{\substack{n\to\infty}}J'_m(u^{\pm}_{m,n})=0\,\,\text{ and }\,\, \lim_{\substack{n\to\infty}}J(u^{\pm}_{m,n})\in \big[\rho,\max_{t\in[0,1]}J(te^\pm)\big].
\end{equation*}
\end{theorem}
\begin{proof}
We define
\begin{equation*}
c^\pm_m:=\inf_{\substack{\gamma\in\Gamma^\pm_m}}\sup_{u\in\gamma([0,1])}J(u),
\end{equation*}
where
\begin{equation*}
\Gamma^\pm_m:=\Big\{\gamma\in C\big([0,1],\overline{\pm D^0_m}\big)\,\,;\,\, \gamma(0)=0,\,\,\,\gamma(1)=e^\pm\Big\}.
\end{equation*}
One can verify easily that the map $\gamma:[0,1]\to\overline{\pm D^0_m}$ defined by $\gamma(t)=te^\pm$ belongs to $\Gamma^\pm_m$.\\
By remarking that $\inf_{\substack{u\in \overline{\pm D^0_m}\\ \|u\|=r}}J(u)\geq \inf_{\substack{u\in X\\ \|u\|=r}}J(u)$, we deduce from $(A_1)$ that $c^\pm_m\geq \rho$. \\
\par We claim that 
\begin{equation}\label{claim2}
\forall\varepsilon_0\in]0,(c^\pm_m-\delta)[,\,\,\exists u\in J^{-1}\big([c^\pm_m-2\varepsilon_0,c^\pm_m+2\varepsilon_0]\big)\cap \left(\overline{\pm D^0_m}\right)\,\,;\,\, \|J'_m(u)\|<\varepsilon_0.
\end{equation}
Indeed, if the claim is not true then there exists $\varepsilon_0\in]0,(c^\pm_m-\delta)/2[$ such that $\|J'_m(u)\|\geq\varepsilon_{0}$ for all $u\in J^{-1}\big([c_m^\pm-2\varepsilon_0,c_m^\pm+2\varepsilon_0]\big)\cap \left(\overline{\pm D^0_m}\right)$. We apply Lemma \ref{deformationlemma} with $c=c^\pm_m$ and $W=V_{\frac{\mu_m}{2}}(\pm P_m)$ and we define
\begin{equation*}
\theta:[0,1]\to \overline{\pm D^0_m},\qquad \theta(t):=\eta(1,\gamma(t)),
\end{equation*}
where $\gamma\in\Gamma^\pm_m$ satisfies
\begin{equation}\label{star}
\sup_{u\in\gamma([0,1])}J(u)\leq c^\pm_m+\varepsilon,
\end{equation}
with $\varepsilon$ and $\eta$ given by Lemma \ref{deformationlemma}.\\
It is not difficult to show, using the properties of $\eta$, that $\theta$ belongs to $\Gamma^\pm_m$.\\
Inequality \eqref{star} and conclusion $(ii)$ of Lemma \ref{deformationlemma} imply that
\begin{align*}
\sup_{\substack{u\in\theta([0,1])}}J(u)&=\sup_{\substack{u\in\eta\big(1,\gamma([0,1])\big)\cap \overline{\pm D^0_m}}}J_m(u)\\
&\leq \sup_{\substack{u\in\eta\big(1,J_m^{-1}(]-\infty,c^\pm_{m}+\varepsilon])\cap \overline{\pm D^0_m}\big)}}J_m(u)\\
&\leq c^\pm_{m}-\varepsilon.
\end{align*}
This contradicts the definition of $c^\pm_m$. Consequently \eqref{claim2} holds and the conclusion of the theorem follows.
\end{proof}

To find sign-changing critical points of mountain pass type, we will use a version of the mountain pass theorem with a new characterization of the mountain pass level given in \cite{LiuLiuWang15}. We will use the following notations:
\begin{equation*}
D_m=D_m^0\cup(-D_m^0)\quad\text{and}\quad S_m:=Y_m\backslash D_m.
\end{equation*}
\begin{theorem}\label{scmpt}
Let $J\in C^1(X,\mathbb{R})$. Assume that for $m>2$ there exists $\mu_m>0$ such that $(A_0)$ is satisfied. Assume also that there exists a continuous map $\varphi_0:\Delta\to Y_m$ satisfying:
\begin{enumerate}
\item[(1)] $\varphi_0(\partial_1\Delta)\subset D_m^0$ and $\varphi_0(\partial_2\Delta)\subset -D_m^0$,
\item[(2)] $\varphi_0(\partial_0\Delta)\cap D_m^0\cap (-D_m^0)=\emptyset$,
\item[(3)] $c_0:=\sup_{u\in\varphi_0(\partial_0\Delta)}J(u)<c_m^\star:=\inf_{u\in\partial(D_m^0)\cap \partial(-D_m^0)}J(u) $,
\end{enumerate}
where 
\begin{equation*}
\Delta=\big\{(s,t)\in\mathbb{R}^2\,\,:\,\, s,t\geq0,\quad s+t\leq1\big\}
\end{equation*}
\begin{equation*}
\partial_1\Delta=\{0\}\times[0,1],\quad \partial_2\Delta=[0,1]\times\{0\}\quad\text{and}\quad \partial_0\Delta=\big\{(s,t)\in \Delta \,\,:\,\, s+t=1\big\}.
\end{equation*}
Then there exists a sequence $(u_m^n)_n\subset V_{\frac{\mu_m}{2}}(S_m)$ such that
\begin{equation*}
\lim_{\substack{n\to\infty}}J'_m(u_m^n)=0\quad\text{and}\quad \lim_{\substack{n\to\infty}}J(u_m^n)\in\big[c_0,\sup_{u\in\varphi_0(\Delta)} J(u)\big].
\end{equation*}
\end{theorem}
Before giving the proof of this theorem, we recall the following useful intersection lemma which the proof can be found in \cite{LiuLiuWang15}.
\begin{lemma}\label{intersection}
If $\varphi:\Delta\to Y_m$ satisfies $\varphi(\partial_1\Delta)\subset D_m^0$, $\varphi(\partial_2\Delta)\subset -D_m^0$, and $\varphi(\partial_0\Delta)\cap D_m^0\cap (-D_m^0)=\emptyset$, then $\varphi(\Delta)\cap\partial(D_m^0)\cap\partial(-D_m^0)\neq\emptyset$.
\end{lemma}
\begin{proof}[{\bf Proof of Theorem \ref{scmpt}.} ]
We define
\begin{equation*}
\Gamma_m:=\big\{\varphi\in C(\Delta,Y_m)\,\,:\,\, \varphi(\partial_1\Delta)\subset D_m^0,\,\,\,\, \varphi(\partial_2\Delta)\subset -D_m^0\,\, \text{and}\,\,\,\, \varphi|_{\partial_0\Delta}=\varphi_0\big\}.
\end{equation*}
It is clear that $\varphi_0\in\Gamma_m$. By Lemma \ref{intersection} above we have $\varphi(\Delta)\cap\partial(D_m^0)\cap\partial(-D_m^0)\neq\emptyset$ for any $\varphi\in \Gamma_m$. This intersection property implies that
\begin{equation*}
c_m:=\inf_{\varphi\in\Gamma_m}\sup_{u\in\varphi(\Delta)\cap S_m}J(u)\geq c_m^\star>c_0.
\end{equation*}
Let us show that
\begin{equation}\label{claimscmpt}
\forall\varepsilon_0\in\big]0,\frac{c_m-c_0}{2}\big[,\,\,\exists u\in J_m^{-1}\left([c_m-2\varepsilon_0,c_m+2\varepsilon_0]\right)\cap V_{\frac{\mu_m}{2}}(S_m)\,\,;\,\, \|J_m'(u)\|<\varepsilon_0.
\end{equation}
Arguing toward a contradiction, we assume that \eqref{claimscmpt} does not hold, that is there exists $\varepsilon_0\in]0,\frac{c_m-c_0}{2}[$ such that $\|J_m'(u)\|\geq\varepsilon_0$, for all $u\in J_m^{-1}\left([c_m-2\varepsilon_0,c_m+2\varepsilon_0]\right)\cap V_{\frac{\mu_m}{2}}(S_m)$. We can then apply Lemma \ref{deformationlemma} with $c=c_m$ and $W=S_m$. Using $\varepsilon$ and the deformation $\eta$ given by Lemma \ref{deformationlemma}, we define the map
\begin{equation*}
\beta:\Delta\to Y_m,\quad x\mapsto \beta(x):=\eta(1,\varphi(x)),
\end{equation*}
where $\varphi\in\Gamma_m$ is chosen such that 
\begin{equation}\label{starscmpt}
\sup_{u\in\varphi(\Delta)\cap S_m}J(u)\leq c_m+\varepsilon.
\end{equation}
Since $c_0<c_m-2\varepsilon$, condition $(3)$ in Theorem \ref{scmpt} and conclusion $(i)$ of Lemma \ref{deformationlemma} imply that $\varphi_0(\partial_0\Delta)\subset J_m^{-1}\left(]-\infty,c_m-2\varepsilon[\right)$. It follows that $\beta\in \Gamma_m$.\\
Now, using the conclusions $(ii)$ and $(iv)$ of Lemma \ref{deformationlemma} and the relation \eqref{starscmpt}, one can easily verify that
\begin{equation*}
\eta\left(1,\varphi(\Delta)\right)\cap S_m\subset \eta\left(1,J_m^{-1}\left(]-\infty,c_m+\varepsilon]\right)\cap S_m\right)\subset J_m^{-1}\left(]-\infty,c_m-\varepsilon]\right).
\end{equation*}
It follows that
\begin{equation*}
\sup_{u\in\beta(\Delta)\cap S_m }J(u)=\sup_{u\in\eta(1,\varphi(\Delta))\cap S_m}J(u)\leq c_m-\varepsilon,
\end{equation*}
which is in contradiction with the definition of $c_m$.\\
The above contradiction assures that \eqref{claimscmpt} holds. We then conclude by letting $\varepsilon_0\to0$.      
\end{proof}

We terminate this section by recalling a version of the symmetric mountain pass theorem we will apply in order to get infinitely many sign-changing critical points. We introduce for $k\geq2$ and $m>k+2$ the following notations:
\begin{equation*}
 Z_k=\overline{\oplus_{j=k}^\infty X_j},\quad Z^m_k=\oplus_{j=k}^m X_j,\quad  B_k:=\big\{u\in Y_k\,;\, \|u\|\leq\rho_k\big\}, 
\end{equation*}
\begin{equation*}
N_k:=\big\{u\in Z_k\,;\,\|u\|=r_k\big\},\,\, N^m_k:=\big\{u\in Z^m_k\,;\,\|u\|=r_k\big\},\,\,\text{where }0<r_k<\rho_k.
\end{equation*} 
The following result was established in \cite{Batkam1}.
\begin{theorem}\label{fountaintheorem}
Let $J\in C^1(X,\mathbb{R})$ be even. Assume that for $k\geq2$ and $m>k+2$, there exist $0<r_k<\rho_k$ and $\mu_m>0$ such that $(A_0)$ and the following two conditions are satisfied:
\begin{enumerate}
\item[$(A'_1)$] $a_k:=\max_{\substack{u\in \partial B_k}}J(u)<b_k:=\inf_{\substack{u\in N_k}}J(u)$.
\item[$(A'_2)$] $N^m_k\subset S_m$.
\end{enumerate}
Then there exists a sequence $(u_{k,m}^n)_n\subset V_{\frac{\mu_m}{2}}(S_m)$ such that
\begin{equation*}
\lim_{\substack{n\to\infty}}J'_m(u_{k,m}^n)=0\quad\text{and}\quad \lim_{\substack{n\to\infty}}J(u_{k,m}^n)\in\big[b_k,\max_{\substack{u\in B_k}}J(u)\big].
\end{equation*}
\end{theorem}

\section{Proof of the main result}\label{section3}

 In this section, we apply the previous abstract theorems to prove our main result. We assume throughout this section that $(f_{1,2,3})$ are satisfied. We will also denote by $|.|_{q}$ the usual norm of the Lebesgue space $L^{q}(\Omega)$. 
 
We define $X$ to be the usual Sobolev space $H_{0}^{1}(\Omega)$ endowed with the inner product
\begin{equation*}
\langle u, v\rangle=\int_{\Omega}\nabla u\nabla vdx
\end{equation*}
and norm $\|u\|^{2}=\langle u, u\rangle$, for $u, v\in H_{0}^{1}(\Omega)$.\\
The following result is well known (see e.g \cite{DM, Ruiz, ZZ}).
\begin{lemma}\label{three}
For each $u\in H_{0}^{1}(\Omega)$, there exists a unique element $\phi_u\in H_{0}^{1}(\Omega)$ such that $-\Delta \phi_u=u^2$. Moreover, $\phi_{u}$ has the following properties: 
\begin{enumerate}
\item[$(i)$] there exists $C>0$ such that $\|\phi_{u}\|\leq C\|u\|^{2}$ and
\begin{equation*}
\int_{\Omega}|\nabla \phi_{u}|^{2}dx=\int_{\Omega}\phi_{u}u^{2}dx\leq C\|u\|^{4};
\end{equation*}
\item[$(ii)$] $\phi_{u}\geq 0$ and $\phi_{tu}=t^{2}\phi_{u}, \ \forall \ t>0;$
\item[$(iii)$] if $u_{n}\rightharpoonup u$ in $H_{0}^{1}(\Omega)$, then $\phi_{u_{n}}\rightharpoonup \phi_{u}$ in $H_{0}^{1}(\Omega)$ and
\begin{equation*}
\lim_{n\to\infty}\int_{\Omega}\phi_{u_{n}}u_{n}^{2}dx=\int_{\Omega}\phi_{u}u^{2}dx.
\end{equation*}
\end{enumerate}
\end{lemma}
We mean by a weak solution of \eqref{2}, a function $u\in H_{0}^{1}(\Omega)$ such that
\begin{equation*}
\left(a+b\|u\|^{2}\right)\int_{\Omega}\nabla u\nabla vdx+\int_{\Omega}\phi_{u}uvdx=\int_{\Omega}f(x, u)vdx,\quad \forall v\in H_{0}^{1}(\Omega).
\end{equation*}
Let us consider the functional $J: H_{0}^{1}(\Omega)\to\mathbb{R}$ defined by 
\begin{equation*}
J(u)=\frac{a}{2}\|u\|^{2}+\frac{b}{4}\|u\|^{4}+\frac{1}{4}\int_{\Omega}\phi_{u}u^{2}dx-\int_{\Omega}F(x, u)dx,\quad u\in H_0^1(\Omega).
\end{equation*}
It is standard to verify that if $(f_1)$ holds then $J$ belongs to $C^1(X,\mathbb{R})$ with 
\begin{equation*}
\big<J'(u),v\big>=\big(a+b\|u\|^2\big)\int_\Omega \nabla u\nabla vdx+\int_{\Omega}\phi_{u}uvdx-\int_{\Omega}f(x, u)vdx,\, \forall u,v\in H_0^1(\Omega).
\end{equation*} 
Clearly, critical points of $J$ are weak solutions of problem \eqref{2}. 

Let us denote by $\lambda_{1}<\lambda_{2}<\lambda_{3}<\ldots$ the distinct eigenvalues of the problem
\begin{equation*}
-\Delta u=\lambda u \ \mbox{in} \ \Omega, \ u=0 \ \mbox{on} \ \partial\Omega.
\end{equation*}
It is well known that each $\lambda_{j}$ has finite multiplicity, the principal eigenvalue $\lambda_{1}$ is simple with positive eigenfunction $e_{1}$, and the eigenfunctions $e_{j}$ corresponding to $\lambda_{j}$ ($j\geq 2$) are sign-changing. Let $X_{j}$ be the eigenspace associated to $\lambda_{j}$. We set for $m>2$
\begin{equation*}
Y_{m}=\oplus_{j=1}^{m}X_{j}, J_{m}:=J|_{Y_{m}}, K_m:=\big\{u\in Y_m\,;\, J'_m(u)=0\big\}\quad\text{and}\quad E_m:=Y_m\backslash K_m
\end{equation*}
\begin{equation*}
+P_{m}:=\big\{u\in Y_{m}; u(x)\geq 0\big\}, Z_{k}^{m}:=\oplus_{j=k}^{m}X_{j}\quad\text{and}\quad N_{k}^{m}:=\big\{u\in Z_{k}^{m}; \|u\|=r_{k}\big\}.
\end{equation*}
Note that, for all $u\in +P_{m}\backslash\{0\}$, we have $\int_{\Omega} ue_{1}dx>0$, while for all $u\in Z_{k}$, $\int_{\Omega} ue_{1}dx=0$, where $e_{1}$ denotes the first eigenfunction of the Laplace operator with Dirichlet condition. This implies that $+P_{m}\cap Z_{k}=\{0\}$. In an analogue way we can show that $-P_{m}\cap Z_{k}=\{0\}$, where $-P_{m}:=\{u\in Y_{m}; u(x)\leq 0\}$. Therefore, since $N_{k}^{m}$ is compact, we conclude that
\begin{equation}\label{six}
\delta_{m}:=dist(N_{k}^{m}, -P_{m}\cup +P_{m})>0.
\end{equation}
Now we construct an odd locally Lipschitz continuous vector field satisfying the condition $(A_{0})$. The idea goes back to \cite{LiuWangZhang}.\\
For $u\in Y_{m}$ fixed, we consider the functional
\begin{equation*}I_{u}(v)=\frac{1}{2}\left(a+b\|u\|^{2}\right)\|v\|^{2}+\frac{1}{2}\int_{\Omega}\phi_{u}v^{2}dx - \int_{\Omega}vf(x, u)dx, \quad v\in Y_{m}.
\end{equation*}

It is not difficult to see that $I_{u}$ is of class $C^{1}$, coercive, bounded from below, weakly lower semicontinuous and strictly convex. Therefore $I_{u}$ admits a unique minimizer $v=Au\in Y_{m}$, which is the unique solution to the problem
\begin{equation*}
-\left(a+b\|u\|^{2}\right)\Delta v+\phi_{u}v=f(x, u), \quad v\in Y_{m}.
\end{equation*}
Clearly, the set of fixed points of $A$ coincides with $K_{m}$. Moreover the operator $A: Y_{m}\to Y_{m}$ has the following properties.

\begin{lemma}\label{field} \quad
\begin{enumerate}
\item[$(1)$] $A$ is continuous and maps bounded sets to bounded sets.
\item[$(2)$] For any $u\in Y_{m}$, we have
\begin{equation}
\langle J'_{m}(u), u-Au\rangle\geq \|u-Au\|^{2},
\end{equation}
\begin{equation}
\|J'_{m}(u)\|\leq (a+ C\|u\|^{2})\|u-Au\|.
\end{equation}

\item[$(3)$] There exists $\mu_{m}\in (0, \delta_{m})$ such that $A(\pm D_{m}^{0})\subset \pm D_{m}^{0}$, where $\delta_{m}$ is defined by $(\ref{six})$.
\end{enumerate}
\end{lemma}
\begin{proof}
(1)\quad Let $(u_{n})\subset Y_{m}$ be such that $u_{n}\to u$ in $Y_{m}$. We set $v_{n}=Au_{n}$ and $v=Au$. Using the definition of $A$, we obtain for any $w\in Y_{m}$
\begin{equation}\label{seven}
\left(a+b\|u_{n}\|^{2}\right)\int_{\Omega}\nabla v_{n}\nabla wdx+\int_{\Omega}\phi_{u_{n}}v_{n}wdx=\int_{\Omega}wf(x, u_{n})dx,
\end{equation}
\begin{equation}\label{eight}
\left(a+b\|u\|^{2}\right)\int_{\Omega}\nabla v\nabla wdx+\int_{\Omega}\phi_{u}vwdx=\int_{\Omega}wf(x, u)dx.
\end{equation}
By choosing $w=v_{n}-v$ in $(\ref{seven})$ and in $(\ref{eight})$, we obtain
\begin{multline*}
\left(a+b\|u_{n}\|^{2}\right)\|v_{n}-v\|^{2}= b\left(\|u_{n}\|^{2}-\|u\|^{2}\right)\int_{\Omega}\nabla v\nabla(v_{n}-v)dx\\
                                                                     + \int_{\Omega}(\phi_{u}v-\phi_{u_{n}}v_{n})(v_{n}-v)dx+\int_{\Omega}(v_{n}-v)(f(x, u_{n})-f(x, u))dx.
\end{multline*}
Observing that
\begin{multline*}
(\phi_{u}v-\phi_{u_{n}}v_{n})(v_{n}-v)=-\phi_{u_{n}}(v_{n}-v)^{2}+(\phi_{u_{n}}-\phi_{u})v(v-v_{n})\\
                                                                  \leq (\phi_{u_{n}}-\phi_{u})v(v-v_{n}),
\end{multline*}
we conclude, from H\"{o}lder inequality and Sobolev embedding theorem, that
\begin{multline*}
\left(a+b\|u_{n}\|^{2}\right)\|v_{n}-v\|^{2}\leq b\left| \|u_{n}\|^{2}-\|u\|^{2}\right|\|v\|\|v_{n}-v\|+ c_{1}|\phi_{u_{n}}-\phi_{u}|_{3}\|v\|\|v_{n}-v\|\\
+c_{2}\big| f(., u_n)-f(., u)\big|_{\frac{p}{p-1}}\|v_{n}-v\|
\end{multline*}
where $c_{1}$ and $c_{2}$ are positive constants. From $(f_{1})$ and Theorem A.2 in \cite{Willem96}, we have $f(., u_{n})-f(., u)\to 0$ in $L^{\frac{p}{p-1}}(\Omega)$. On the other hand, from definition of $\phi_{u_{n}}$ and $\phi_{u}$, we obtain $\phi_{u_{n}}-\phi_{u}\to 0$ in $L^{3}(\Omega)$. Hence, $v_{n}\to v$ in $Y_{m}$, showing that $A$ is continuous.

To finish this item, we observe that taking $v=w=Au$ in \eqref{eight} leads to
\begin{equation*}
\left(a+b\|u\|^{2}\right)\|Au\|^{2}+\int_{\Omega}\phi_{u}(Au)^{2}dx=\int_{\Omega}Au f(x, u)dx,
\end{equation*}
which implies, using $(f_1)$ and the Sobolev embedding theorem, that
\begin{equation*}
a\|Au\|\leq c_{2}(1+\|u\|^{p}),
\end{equation*}
where $c_{2}$ is a positive constant. Therefore, $Au$ is bounded whenever $u$ is bounded.\\
(2)\quad  Taking $w=u-Au$ in $(\ref{eight})$, it follows that
\begin{multline*}
\left(a+b\|u\|^{2}\right)\int_{\Omega}\nabla Au\nabla (u-Au)dx+\int_{\Omega}\phi_{u}Au(u-Au)dx=
\int_{\Omega}(u-Au) f(x, u)dx,
\end{multline*}

whence
\begin{equation*}
\langle J'_{m}(u), u-Au\rangle=\left(a+b\|u\|^{2}\right)\|u-Au\|^{2}+\int_{\Omega}\phi_{u}(u-Au)^{2}dx\geq a\|u-Au\|^{2}.
\end{equation*}
Moreover, using again $(\ref{eight})$, we have
\begin{eqnarray*}
\langle J'_{m}(u), w\rangle&=&\left(a+b\|u\|^{2}\right)\int_{\Omega}\nabla u\nabla wdx+\int_{\Omega}\phi_{u}uwdx-\int_{\Omega}wf(x, u)dx\\
                                               &=& \left(a+b\|u\|^{2}\right)\int_{\Omega}\nabla (u-Au)\nabla wdx+\int_{\Omega}\phi_{u}(u-Au)wdx.
\end{eqnarray*}
Applying H\"{o}lder inequality and Sobolev embedding theorem, we conclude that
\begin{equation*}
\|J'_{m}(u)\|\leq (a+ C\|u\|^{2})\|u-Au\|,
\end{equation*}
for some constant $C>b$.\\
(3)\quad From $(f_{1})$ and $(f_{2})$, for any $\varepsilon>0$ there exists $c_\varepsilon>0$ such that
\begin{equation}\label{festimate}
|f(x, t)|\leq \varepsilon|t|+c_\varepsilon|t|^{p-1}, \ \forall \ t\in \mathbb{R}.
\end{equation}
Let $u\in Y_{m}$ and $v=Au$. We denote $w^+=\max\{0, w\}$ and $w^-=\min\{0, w\}$, for any $w\in X$. Taking $w=v^{+}$ in $(\ref{eight})$ and using H\"{o}lder inequality, we obtain
\begin{align*}
\left(a+b\|u\|^{2}\right)\|v^{+}\|^{2}+\int_{\Omega}\phi_{u}(v^{+})^{2}dx&=\int_{\Omega}v^{+}f(x, u)dx\\
&\leq \varepsilon|u^{+}|_{2}|v^{+}|_{2}+c_\varepsilon|u^{+}|_{p}^{p-1}|v^{+}|_{p},
\end{align*} 
which implies
\begin{equation}\label{ten}
\|v^{+}\|^{2}\leq \frac{1}{a}\left(\varepsilon|u^{+}|_{2}|v^{+}|_{2}+c_\varepsilon|u^{+}|_{p}^{p-1}|v^{+}|_{p}\right).
\end{equation}
Since $|z^{+}|_{s}\leq |z-w|_{s}$, for all $z\in X$, $w\in -P$, and $1\leq s\leq 2^{\ast}$,
there exists a positive constant $c_1=c_1(s)$ such that $|u^{+}|_{s}\leq c_1dist(u, -P)\leq c_1dist(u, -P_{m})$.
 Note that $dist(v, -P_{m})\leq \|v^{+}\|$ and, consequently, by $(\ref{ten})$ and Sobolev embeddding theorem that
\begin{equation*}
dist(v, -P_{m})\|v^{+}\|\leq \|v^{+}\|^{2}\leq c_2\left( \varepsilon dist(u, -P_{m})+c_\varepsilon dist(u, -P_{m})^{p-1}\right),
\end{equation*}
where $c_2$ is a positive constant. Therefore,
\begin{equation*}
dist(v, -P_{m})\leq c_2\left(\varepsilon dist(u, -P_{m})+c_\varepsilon dist(u, -P_{m})^{p-1}\right).
\end{equation*}
In the same way, we can prove that
\begin{equation*}
dist(v, +P_{m})\leq c'_2\left( \varepsilon dist(u, +P_{m})+c_\varepsilon dist(u, +P_{m})^{p-1}\right).
\end{equation*}
where $c'_2$ is also a positive constant.\\
Hence
\begin{equation*}
dist(v, \pm P_{m})\leq c_3\left(\varepsilon dist(u, \pm P_{m})+c_\varepsilon dist(u, \pm P_{m})^{p-1} \right),
\end{equation*}
where $c_3=\max\{c_2,c_2'\}$. 
Fixing $\varepsilon$ small enough, we can choose $\mu_m$ such that 
\begin{equation}\label{mum}
0<\mu_m<\min\Big\{\delta_m,\frac{1}{m}\Big\},
\end{equation}
and 
\begin{equation*}
dist(v, \pm P_{m})\leq \frac{1}{2}dist(u, \pm P_{m})\quad\text{whenever}\quad dist(u, \pm P_{m})<\mu_{m}.
\end{equation*}
It then follows that $A(\pm D_{m}^{0})\subset \pm D_{m}^{0}$.
 \end{proof}

Using $\mu_{m}$ as above, we define
\begin{equation}\label{dmzero}
\pm D_{m}^{0}:=\{u\in Y_{m}: dist(u, \pm P_{m})<\mu_{m}\}, 
\end{equation}
\begin{equation*}
D_{m}=D_{m}^{0}\cup (-D_{m}^{0}) \ \text{ and } \ S_{m}:=Y_{m}\backslash D_{m}.
\end{equation*}
Observe that we can not yet insure that the vector field $A: Y_{m}\to Y_{m}$ is locally Lipschtiz continuous. However, as pointed out in \cite{LiuWangZhang}, we can argue as in \cite{BartschLiu04} and use $A$ to construct another vector field which satisfies the condition $(A_{0})$. More precisely, we have the following result.

\begin{lemma}\label{Lip}
There exists a locally Lipschitz continuous operator $B:E_{m}\to Y_{m}$ such that
\begin{enumerate}

\item[$(1)$] $\langle J'(u), u- Bu\rangle\geq \frac{1}{2}\|u-Au\|^{2}$, for any $u\in E_{m}$.

\item[$(2)$] $\frac{1}{2}\|u-Bu\|\leq \|u-Au\|\leq 2\|u-Bu\|$, for any $u\in E_{m}$.

\item[$(3)$] $B((\pm D_{m}^{0})\cap E_{m})\subset \pm D_{m}^{0}$.

\end{enumerate}
Moreover, if $A$ is odd then so is $B$.
\end{lemma}

{\bf Proof.} Follows the same steps from \cite{BartschLiu04}, see also \cite{Batkam1}.

\begin{remark}\label{rem1}
Lemmas $\ref{field}$ and $\ref{Lip}$ imply that
\begin{equation*}
\langle J'(u), u- Bu\rangle\geq \frac{1}{8}\|u-Au\|^{2}\quad\text{and}\quad \|J'_{m}(u)\|\leq 2(a+ C\|u\|^{2})\|u-Au\|.
\end{equation*}
\end{remark}

\begin{lemma}\label{lastcondition}
Let $\rho_{1}<\rho_{2}$ and $\alpha>0$. Then there exists $\beta>0$ such that $\|u-B(u)\|\geq\beta$ if $u\in Y_m$ is such that $J_m(u)\in[\rho_{1},\rho_{2}]$ and $\|J'_m(u)\|\geq\alpha$.
\end{lemma}

\begin{proof}
From the definition of operator $A$, it follows that for each $u\in Y_{m}$
\begin{equation*}
\left(a+b\|u\|^{2}\right)\int_{\Omega}\nabla Au\nabla udx+\int_{\Omega}\phi_{u}uAudx=\int_{\Omega}uf(x, u)dx .
\end{equation*}
Whence,
\begin{multline}
J_{m}(u)-\frac{1}{\mu}\left(a+b\|u\|^{2}\right)\int_{\Omega}\nabla u\nabla(u-Au)dx-\frac{1}{\mu}\int_{\Omega}\phi_{u}u(u-Au)dx=\\
a\left(\frac{1}{2}-\frac{1}{\mu}\right)\|u\|^{2}+b\left(\frac{1}{4}-\frac{1}{\mu}\right)\|u\|^{4}+\left(\frac{1}{4}-\frac{1}{\mu}\right)\int_{\Omega}\phi_{u}u^{2}dx\\
+\int_{\Omega}\left(\frac{1}{\mu}f(x, u)u-F(x, u)\right)dx.
\end{multline}

Using $(f_{3})$ and Lemma $\ref{Lip}(2)$, we have
\begin{align}\label{eleven}
\nonumber a\left(\frac{1}{2}-\frac{1}{\mu}\right)\|u\|^{2}+b\left(\frac{1}{4}-\frac{1}{\mu}\right)\|u\|^{4}
                                                                         &\leq |J_{m}(u)|+\frac{1}{\mu}\left(a+b\|u\|^{2}\right)\|u\|\|u-Au\|\\ 
                                                                         &\leq |J_{m}(u)|+\frac{2}{\mu}\left(a+b\|u\|^{2}\right)\|u\|\|u-Bu\|.
\end{align}
Suppose that there exists a sequence $(u_{n})\subset Y_{m}$ such that $J_{m}(u_{n})\in [\rho_{1}, \rho_{2}], \|J_{m}'(u_{n})\|\geq \alpha$ and $\|u_{n}-Bu_{n}\|\to 0$. From $(\ref{eleven})$, we conclude that $(u_{n})$ is bounded. Finally, by Remark $\ref{rem1}$ we derive $J_{m}'(u_n)\to 0$, which is a contradiction.
\end{proof}

\subsection{Existence of constant sign solutions}
In this subsection, we prove the existence of a positive solution and a negative solution to \eqref{2}.\\
The following lemma shows that $J_m$ has the mountain pass geometry.

\begin{lemma}\label{geometry}
There are $e^{\pm}\in \pm P_{2}$ and $r>0$ such that the condition $(A_{1})$ in Theorem $\ref{mountain}$ holds.
\end{lemma}
\begin{proof}
Let $\sigma:=\inf\{|\nabla u|_2^2/|u|_2^2\,;\,u\in H_0^1(\Omega)\}$. For $\varepsilon=a\sigma/2$ in \eqref{festimate}, there exists $c_{1}>0$ such that 
\begin{equation*}
|F(x, t)|\leq \frac{a\sigma}{4}|t|^{2}+\frac{c_{1}}{p}|t|^{p}, \ \forall \ t\in\mathbb{R}.
\end{equation*}
The last inequality together with Lemma \ref{three}(ii) and Sobolev embedding theorem gives us
\begin{equation*}
J(u)\geq \frac{a}{4}\|u\|^{2}-\frac{c_{2}}{p}\|u\|^{p},
\end{equation*}
where $c_2>0$ is constant. Thus, by choosing $r=(a/2c_{2})^{\frac{1}{p-2}}>0$, we obtain 
\begin{equation*}
J(u)\geq\left(\frac{1}{2}-\frac{1}{p}\right)\left(\frac{a}{2}\right)^{p/(p-2)}\frac{1}{c_{2}^{2/(p-2)}}=:c_{\ast},
\end{equation*}
for all $u\in H_{0}^{1}(\Omega)$ with $\|u\|=r$.

On the other hand, from $(f_{3})$ there are positive constants $c_{3}$ and $c_{4}$ such that
\begin{equation*}
F(x, t)\geq c_{3}|t|^{\mu}-c_{4}.
\end{equation*}
Hence, fixing $e_{\ast}^{\pm}\in \pm P_{2}\backslash\{0\}$, it follows from Lemma \ref{three}-(ii) that
\begin{equation*}
J(te_{\ast}^{\pm})\leq\left(\frac{a}{2}\|e_{\ast}^{\pm}\|^2\right)t^{2}+\left(\frac{b}{4}\|e_{\ast}^{\pm}\|^4+\frac{1}{4}\int_\Omega\phi_{e_{\ast}^{\pm}}(e_{\ast}^{\pm})^2\right)t^{4}-\left(c_3|e_{\ast}^{\pm}|^\mu_\mu\right)t^{\mu}+c_4|\Omega|,
\end{equation*}
for all $t>0$. Since $\mu>4$, we can choose $t_{\ast}>0$ large enough such that, defining $e^{\pm}:=t_{\ast}e_{\ast}^{\pm}$, we have $J(e^{\pm})<0$ and $\|e^{\pm}\|>r$. This shows that
\begin{equation*}
\rho=\inf_{u\in X\atop_{\|u\|=r}}J(u)\geq c_{\ast}>0=\max\{J(0), J(e^{\pm})\}.
\end{equation*}
\end{proof}

\begin{proof}[{\bf Proof of Main Theorem (Part 1).}]
  By Lemmas $\ref{geometry}$, $\ref{Lip}$ and $\ref{lastcondition}$, and Remark $\ref{rem1}$ the conditions $(A_{0})$ and $(A_{1})$ of Theorem $\ref{mountain}$ are satisfied.

 Applying Theorem $\ref{mountain}$ we find sequences $\{u_{m,n}^\pm\}_n\subset \overline{\pm D^0_m}$ such that 
\begin{equation*}
\lim_{\substack{n\to\infty}}J'_m(u^{\pm}_{m,n})=0\,\,\text{ and }\,\, \lim_{\substack{n\to\infty}}J(u^{\pm}_{m,n})\in \big[\rho,\max_{t\in[0,1]}J(te^\pm)\big].
\end{equation*}
For any $u\in Y_m$ we have, in view of $(f_3)$, Lemma \ref{three}(ii) and since $\mu>4$, 
\begin{equation}\label{pss}
 J(u)-\frac{1}{\mu}\big<J_m'(u),u\big> \geq a\left(\frac{1}{2}-\frac{1}{\mu}\right)\|u\|^2+b\left(\frac{1}{4}-\frac{1}{\mu}\right)\|u\|^4.
\end{equation}
It follows from \eqref{pss} that the sequences $\{u_{m,n}^\pm\}_n$ are bounded in $Y_m$. Since $Y_m$ is of finite dimension we have, up to subsequences, $u_{m,n}^\pm\to u_m^\pm$ in $Y_m$ as $n\to\infty$. Since $\overline{\pm D^0_m}$ is closed and $J_m$ is smooth, it follows that 
\begin{equation*}
u_{m}^\pm\in \overline{\pm D^0_m},\quad J_m'(u_m^\pm)=0,\quad\text{and}\quad J(u^{\pm}_{m})\in \big[\rho,\max_{t\in[0,1]}J(te^\pm)\big]
\end{equation*}
Inequality \eqref{pss} above then implies that the sequences $\{u_m^\pm\}_m$ are bounded in $X$. Going to subsequences if necessary we can assume that $u_m^\pm\rightharpoonup u^\pm$ in $X$ and $u_m^\pm\to u^\pm$ in $L^p(\Omega)$ and in $L^q(\Omega)$, as $m\to\infty$. We denote by $\pi_m$ the orthogonal projection of $X$ onto $Y_m$. We have
\begin{multline}\label{eee}
\big<J_m'(u_m^\pm),u_m^\pm-\pi_m u^\pm\big>=\big(a+b\|u_m^\pm\|^2\big)\big<u_m^\pm,u_m^\pm-\pi_{m}u^\pm\big>+\int_\Omega \phi_{u_m^\pm}u_m^\pm(u_m^\pm-\pi_m u^\pm)dx \\
-\int_\Omega \big(u_m^\pm-\pi_mu\big)f(x,u_m^\pm)dx.
\end{multline}
The H\"{o}lder inequality gives
\begin{equation*}
\big|\int_\Omega \big(u_m^\pm-\pi_mu\big)f(x,u_m^\pm)dx\big|\leq|u_m^\pm-\pi_mu^\pm|_p|f(x,u_m^\pm|_{\frac{p}{p-1}}
\end{equation*}
\begin{equation*}
\big|\int_\Omega \phi_{u_m^\pm}u_m^\pm(u_m^\pm-\pi_m u^\pm)dx \big|\leq|\phi_{u_m^\pm}|_3|u_m^\pm|_3|u_m^\pm-\pi_m u^\pm|_3.
\end{equation*}
Since the sequences $\{u_m^\pm\}_m$ are bounded, we deduce from $(f_1)$ and Lemma \ref{three}(i) that $\{\big|f(x,u_m^\pm)\big|_{\frac{p}{p-1}}\}_m$ and $\{\big|\phi_{u_m^\pm}\big|_3\}_m$ are bounded. One can then easily deduce from \eqref{eee}, since $J_m'(u_m^\pm)=0$ and $\pi_mu^\pm\to u^\pm$ in $X$, that $u_m^\pm\to u^\pm$ in $X$ as $m\to\infty$. It is easy to see that $u^\pm$ are critical points of $J$. \\
It remains to show that $u^+$ is positive and that $u^-$ is negative.\\
 We first remark that $J(u^\pm)\geq\rho>0$ implies that $u^\pm\neq0$. \\
Now, recall the definition \eqref{dmzero} of $\pm D^0_m$. Since $u_m^\pm\in \overline{\pm D^0_m}$ and $\pm P_m\subset \pm P$, we have 
\begin{equation}\label{estimate}
dist(u_m^\pm,\pm P)\leq dist(u_m^\pm,\pm P_m)\leq \mu_m,
\end{equation}
where $\mu_m$ is given by \eqref{mum}. Since $\mu_m\to0$ as $m\to\infty$, we conclude by taking the limit $m\to\infty$ in \eqref{estimate} that $u^\pm\in\pm P$.
\end{proof}

\subsection{Existence of a sign-changing solution}
We show in this subsection that \eqref{2} has a sign-changing solution of mountain pass type under assumptions $(f_{1,2,3})$. The next two lemmas will be very helpful.
\begin{lemma}\label{kappa}
For $q\in[2,6]$ there exists $\kappa_q>0$ independent of $\mu_m$ such that 
\begin{equation*}
|u|_q\leq \kappa_q\mu_m,\quad \forall u\in D^0_m\cap(-D_m^0).
\end{equation*}
\end{lemma}
\begin{proof}
This follows from the fact that there exists $\kappa_q>0$ such that for any $u\in Y_m$
\begin{equation*}
|u^\pm|_q=\inf_{v\in \mp P_m}|u-v|_q\leq \kappa_q\inf_{v\in \mp P_m}\|u-v\|=\kappa_q dist(u,\mp P_m).
\end{equation*}
\end{proof}
\begin{lemma}\label{cstar}
For $m$ large enough we have 
\begin{equation*}
J(u)\geq \frac{a}{8}\mu_m^2\quad\text{for}\quad u\in \partial(D_m^0)\cap \partial(-D_m^0).
\end{equation*}
\end{lemma}
\begin{proof}
 Let $u\in \partial(D_m^0)\cap \partial(-D_m^0)$. It is clear that $\|u^\pm\|\geq dist(u,\mp P_m)=\mu_m$. Taking $\varepsilon=\frac{a}{2\kappa_2^2}$ in \eqref{festimate} and using Lemma \ref{kappa}, we see that
\begin{equation*}
J(u)\geq\frac{a}{2}\|u\|^2-\int_\Omega F(x,u)dx\geq \frac{a}{4}\mu_m^2-c\mu_m^p,
\end{equation*}
where $c>0$ is a constant. We conclude by using the fact that $\mu_m\to0$ as $m\to\infty$.
\end{proof}
In the following proof we adopt the notations of Theorem \ref{scmpt}.
\begin{proof}[{\bf Proof of Main Theorem  (Part 2).} ]
 We first follow \cite{LiuWangZhang} to verify the assumptions of Theorem \ref{scmpt}. We define the continuous map $\varphi_0:\Delta\to Y_m$ by $\varphi_0(s,t)=R(se_2^-+te_2^+)$ for all $(s,t)\in\Delta$, where $e_2$ is an eigenfunction corresponding to the second eigenvalue of the Laplacian, $e_2^-=\min\{e_2,0\}$, $e_2^+=\max\{e_2,0\}$, and $R>0$ is a constant to be determined later. Since $e_2$ is sign-changing, $e_2^\pm$ are not equal to $0$. Obviously, $\varphi_0(0,t)\in D_m^0$ and $\varphi_0(s,0)\in -D_m^0$. \\
Now a simple computation shows that $\delta:=\min\{|(1-t)e^{-}_2+te_2^+|_2\, :\,t\in[0,1]\}>0$. Then, $|u|_2\geq\delta R$ for $u\in\varphi_0(\partial_0\Delta)$ and it follows from Lemma \ref{kappa} that $\varphi_0(\partial_0\Delta)\cap D_m^0\cap(-D_m^0)=\emptyset$ for $R$ large enough. \\
Recall that $(f_3)$ implies $F(x,u)\geq c_1|u|^\mu-c_2$ for some positive constants $c_1$ and $c_2$. Then we obtain using Lemma \ref{three}(i)
\begin{equation*}
J(u)\leq \frac{a}{2}\|u\|^2+C_1\|u\|^4-C_2|u|_\mu^\mu+C_3,
\end{equation*}
where $C_{1,2,3}$ are positive constants. This inequality together with Lemma \ref{cstar} implies that, for $m$ and $R$ large enough,
\begin{equation*}
\sup_{u\in\varphi_0(\partial_0\Delta)}J(u)<0<c_m^\star.
\end{equation*}
We have then shown that the assumptions of Theorem \ref{scmpt} are satisfied. Hence, applying Theorem \ref{scmpt} we find, for $m$ large enough, a sequence $(u_m^n)_n\subset V_{\frac{\mu_m}{2}}(S_m)$ such that
\begin{equation*}
\lim_{\substack{n\to\infty}}J'_m(u_m^n)=0\quad\text{and}\quad \lim_{\substack{n\to\infty}}J(u_m^n)\in\big[c_0,\sup_{u\in\varphi_0(\Delta)} J(u)\big].
\end{equation*}
We can proceed as in the proof of Part 1 to show that, up to a subsequence, $u_m^n\to u_m$ in $Y_m$ as $n\to\infty$ and that $u_m\in V_{\frac{\mu_m}{2}}(S_m)$, $J'_m(u_m)=0$, and the sequence $(u_m)_m$ converges, up to a subsequence, to a critical point $u$ of $J$ as $m\to\infty$. To see that $u$ is sign-changing, we first observe that
\begin{equation*}
\big<J'_m(u_{m}),u_{m}^{\pm}\big>=0\,\,\Rightarrow\,\, a\|u_{m}^\pm\|^2\leq\int_\Omega u_{m}^\pm f(x,u_{m}^\pm).
\end{equation*}
 We recall that $(f_1)$ and $(f_2)$ imply 
\begin{equation*}
\forall\varepsilon>0,\quad\exists c_\varepsilon>0\,\,;\,\, |f(x,t)|\leq \varepsilon|t|+c_\varepsilon|t|^{p-1},\quad\forall (x,t)\in\Omega\times\mathbb{R}.
\end{equation*}
We then obtain by using the Sobolev embedding theorem
\begin{equation*}
a\|u_{m}^\pm\|^2\leq \int_\Omega u_{m}^\pm f(x,u_{m}^\pm)\leq c\big(\varepsilon\|u_{m}^\pm\|^2+c_\varepsilon\|u_{m}^\pm\|^p\big),
\end{equation*} 
for some constant $c>0$.  Since $u_{m}$ is sign-changing, $u_{m}^\pm$ are not equal to $0$. Choosing $\varepsilon$ small enough it follows that $\|u_{m}^\pm\|\geq \alpha>0$, where $\alpha$ does not depend on $m$. By passing to the limit $m\to\infty$, we conclude that $u$ is sign-changing.
\end{proof}
 
\subsection{Existence of high energy sign-changing solutions}
In this subsection, we show that \eqref{1} has infinitely many large energy sign-changing solutions by applying Theorem \ref{fountaintheorem}. We recall the definition of $Y_k$ and $Z_k$ ($k\geq2$ ). 
\begin{equation*}
  Y_k=\overline{\oplus_{j=1}^k X_j}\quad \text{and}\quad Z_k=\overline{\oplus_{j=k}^\infty X_j},
\end{equation*}
where $X_j$ designates, as above, the eigenspace corresponding to the $j$th eigenvalue of the Laplacian.
\begin{proof}[{\bf Proof of Main Theorem (Part 3).}]
 Since the argument is the same as in \cite{Batkam1,Batkam2}, we just provide a sketch here.
 
Using $(f_1)$ and Lemma \ref{three}(ii), we obtain
\begin{equation*}
J(u)\geq \frac{a}{2}\|u\|^2-c_1|u|_p^p-c_2, \quad\forall u\in X,
\end{equation*}
where $c_1,c_2>0$ are constant. It then follows that for any $u\in Z_k$ such that
\begin{equation*}
\|u\|=r_k:=\left(\frac{c_1}{a}p\beta_k^p\right)^{\frac{1}{2-p}},
\end{equation*} 
we have 
 \begin{equation*}
 J(u)\geq a\left(\frac{1}{2}-\frac{1}{p}\right)\left(\frac{c_1}{a}p\beta_k^p\right)^{\frac{2}{2-p}}-c_2,
 \end{equation*}
where
\begin{equation*}
\beta_k:=\sup_{\substack{v\in Z_k\\\|v\|=1}}|v|_p.
\end{equation*}
 Then we obtain
By Lemma $3.8$ in \cite{Willem96}, $\beta_k\to0$ and then $r_k\to\infty$, as $k\to\infty$. 

On the other hand, using the fact that $Y_k$ is finite-dimensional one can easily verify that $J(u)\to-\infty$, as $\|u\|\to\infty$, $u\in Y_k$. Therefore, for $k$ big enough we can choose $\rho_k>r_k$ such that the conditions of Theorem \ref{fountaintheorem} are satisfied. We can then proceed as in \cite{Batkam1,Batkam2} to show that $J$ has a sequence $\{u_k\}$ of sign-changing critical points such that $J(u_k)\to\infty$, as $k\to\infty$.
\end{proof}

\vskip0.5cm \noindent {\bf Acknowledgements.}  CJB was supported by The Fields Institute for Research in Mathematical Sciences and The Perimeter Institute for Theoretical Physics through a Fields-Perimeter Africa Postdoctoral Fellowship.

\end{document}